\documentclass[reqno, a4paper, 10pt]{amsart}
\usepackage{amssymb}
\usepackage{mathrsfs}
\usepackage{amssymb, url, color, graphicx, amscd, mathrsfs}
\usepackage[colorlinks=true, bookmarks=true, pdfstartview=FitH, pagebackref=true, linktocpage=true, linkcolor = magenta, citecolor = blue]{hyperref}
\usepackage{graphicx}
\usepackage{times}

\newtheorem{theorem}{Theorem}[section]
\newtheorem{lemma}[theorem]{Lemma}
\newtheorem{corollary}[theorem]{Corollary}

\theoremstyle{definition}

\newtheorem{example}[theorem]{Example}

\theoremstyle{remark}
\newtheorem{remark}[theorem]{Remark}

\numberwithin{equation}{section}

\DeclareMathOperator{\Ric}{Ric}

\allowdisplaybreaks

\begin{document}

\title[Gradient estimate]{Gradient estimates and Liouville type theorems for Poisson equations}

\def\cfac#1{\ifmmode\setbox7\hbox{$\accent"5E#1$}\else\setbox7\hbox{\accent"5E#1}\penalty 10000\relax\fi\raise 1\ht7\hbox{\lower1.0ex\hbox to 1\wd7{\hss\accent"13\hss}}\penalty 10000\hskip-1\wd7\penalty 10000\box7 }

\author[N.T. Dung]{Nguyen Thac Dung}
\address[N.T. Dung]{Department of mathematics, Hanoi University of Science (VNU), Ha N\^{o}i, Vi\^{e}t Nam}
\email{\href{mailto: N.T. Dung <dungmath@gmail.com>}{dungmath@gmail.com}}

\author[N.N. Khanh]{Nguyen Ngoc Khanh}
\address[N.N. Khanh]{Department of mathematics, University of Jyv\"{a}skyl\"{a} - Jyv\"{a}skyl\"{a}, Finland.}
\email{\href{mailto: N.N. Khanh <khanh.mimhus@gmail.com>}{khanh.mimhus@gmail.com}}


\begin{abstract}
In this paper, we will address to the following parabolic equation
$$
u_t=\Delta_fu + F(u)
$$
on a smooth metric measure space with Bakry-\'{E}mery curvature bounded from below. Here $F$ is a differentiable function defined in $\mathbb{R}$. Our motivation is originally inspired by gradient estimates of Allen-Cahn and Fisher equations (\cite{Bai17, CLPW17}). In this paper, we show new gradient estimates for these equations. As their applications, we obtain Liouville type theorems for positive or bounded solutions to the above equation when either $F=cu(1-u)$ (the Fisher equation) or; $F=-u^3+u$ (the Allen-Cahn equation); or $F=au\log u$ (the equation involving gradient Ricci solitons).
\end{abstract}



\maketitle

\section{Introduction}
Let $(M, g, e^{-f}dv)$ be a smooth metric measure space, namely $(M, g)$ is a Riemnnian manifold with metric tensor $g$, $f$ is a smooth function on $M$ and $dv$ is the volume element with respect to $g$. On $M$, the weighted Laplacian is defined by
$$\Delta_f\cdot:=\Delta-\left\langle\nabla f, \nabla\cdot\right\rangle$$
and the Bakry-\'Emery is given by
$$Ric_f:=Ric+Hess f,$$
where $Ric$ is the Ricci tensor on $M$. 
We consider the following general nonlinear parabolic equation
\begin{equation}\label{general}
u_t=\Delta_f u+F(u).
\end{equation}
Here $F$ is differential function in $\mathbb{R}$ and $u$ is assumed to be smooth. Depending on the given function $F$, the equation \eqref{general} describes several important physical and mathematical phenonmenons, such as the equation of gradient Ricci solitons, Yamabe equations, Lichnerowic equation, ect. We refer the reader to \cite{Bai17, DKN18, RR95, SZ06, Wu15} for further discussion.
If $u$ is a standing solution to \eqref{general} that is $u_t=0$, then we have $\Delta_fu+F(u)=0$. This equation is called Poisson equation. When $F(u)=-cu^2+cu, c$ is a positive constant ($F(u)=-u^3+u$), the equation is said to be Fisher equation or Fisher-KKP equation (Allen-Cahn equation), respectively. Originally, the Allen-Cahn equation appeared in the study of the process of phase separation in iron alloys, including order-disorder transitions (see \cite{CCK15, Bai17}. It is also related to the investigate of minimal surface, therefore it is an interesting topic for differential geometry too (see \cite{dPKW13, Pac12} and the references therein). The equation $\Delta u-cu^2+cu=0, c\in\mathbb{R}^+$ was proposed by R. A. Fisher in 1937 to describe the propagation of an evolutionarily advantageous gene in a population \cite{fisher}, and was also independently described in a seminal paper by A. N. Kolmogorov, I. G. Petrovskii, and N. S. Piskunov in the same year \cite{Kol}; for this reason, it is often referred to in the literature as the Fisher-KPP equation. We refer the reader to \cite{CLPW17, fisher, Kol} for further details and additional references therein.

In this paper, motivated by gradient estimates for heat equation given in \cite{Bai17, CLPW17, DKN18, LY86, SZ06, Wu15}, we will study gradient estimates for the general parabolic equation \eqref{general}. Our main purpose is to derive a new gradient estimate for \eqref{general} then use is to show Liouville properties for Fisher and Allen-Cahn equation. The main theorem in our setting is as follows.
\begin{theorem}\label{thmGE-II}
Let $(M^n, g, e^{-f}dv)$ be a smooth metric measure space with $\Ric_f\geqslant -(n-1)K$. Suppose that $u$ is a smooth solution of equation
$$u_t=\Delta_fu+F(u)$$ 
on $Q_{R, T}:=B(x_0,R)\times[t_0-T,t_0]\subset M\times[0,\infty)$. If $\varepsilon,m,M$ are fixed constants such that $0<\varepsilon<1,0<m\leqslant u\leqslant M$ in $Q_{R, T}$ then there exists a constant $c$ depending only on $n$ such that 
\begin{equation}
\varepsilon u^{\varepsilon-1}|\nabla u|\leq c(n)M^\varepsilon\sqrt{\frac{\varepsilon}{1-\varepsilon}}\left (\frac{1}{R}+\sqrt{\frac{1-\varepsilon}{\varepsilon}}\frac{1}{R}+\frac{\sqrt{\alpha}}{\sqrt{R}}+\frac{1}{\sqrt{t-t_0+T}}+\sqrt{\max\limits_{Q_{R, T}}\{(n-1)K+H, 0\}}\right )\label{theorem1}
\end{equation}
in $Q_{R/2,T}$ with $t\neq t_0-T$, and
$$
H= (\varepsilon-1)F/u+F'
$$
and $\alpha:=\max_{x\in B(x_0,1)}\Delta_fr(x).$
\label{theorem23}
\end{theorem}
We think that this kind of gradient estimate has a potential to study many interesting equation such as the equations given in \cite{DKN18, Wu15},.... This estimate can be considered as a gradient estimate of Li-Yau type. It may be of interest to find gradient estimates of Souplet-Zhang type and its applications (\cite{SZ06}) for \eqref{general}. As applications of Theorem \ref{thmGE-II}, we prove several Liouville type theorems. The first one is given on smooth metric measure spaces as belows. 
\begin{corollary}\label{corol1}
Let $(M, g, e^{-f}dv)$ be a smooth metric measure space with $Ric_f\geq0$. For $a\in\mathbb{R}$, consider the heat equation
\begin{equation}\label{grs1}
u_t=\Delta_fu+au\log u.
\end{equation}
Then the following statements hold true.
\begin{enumerate}
\item If $a>0$ and there does not exist positive solution $u(x, t)$  equation \eqref{grs1} such that $u\leq c<e^{-1}$.
\item If $a<0$ and $u$ is a positive solution to equation \eqref{grs1} then $u$ does not exist provided that $0<c\leq u(x, t)\leq D<1$. Moreover if $0<c\leq u$ and $u$ is of polynomial growth, namely $u(x, t)=o(r^{N_1}(x)+t^{N_2})$ for some $N_1, N_2\in \mathbb{N}$ then $u\equiv1$.
\item If $a=0$ and $u$ is a positive ancient solution to \eqref{grs1} then $u$ is constant provided that $u(x, t)=o(r^{N_1}(x)+t^{N_2})$ for some $N_1, N_2\in \mathbb{N}$.
\end{enumerate}
\end{corollary}
This corollary is a significant improvement of Theorem 1.3 in \cite{Wuna}. In fact, the author, in \cite{Wuna} required that the bound of $u$ is $e^{-2}$. Moreover, the rate of $u$ given there is of sublinear growth. 

In \cite{RR95}, the authors proved in Theorem 2 that on a Riemannian manifold with $Ric_M\geq0$ if the Allen-Cahn equation $\Delta u-u^3+u=0$ has a solution $u$ satisfying $3u^2-1\geq0$ and $u$ is of sublinear growth on $M$ then $u$ is constant. Using Theorem \ref{thmGE-II}, we can improve their results by assuming $u^2\geq m^2>0$, for some $m>0$. Note that in \cite{RR95}, the authors required that $u^2\geq\frac{1}{3}$. Our improvement is stated in the following theorem.
\begin{theorem}\label{ratto}
Let $M$ be a Riemannian manifold with $Ric_M\geq0$. Suppose that $u$ is a solution to 
\begin{equation}\label{allen1}
\Delta u-u^3+u=0
\end{equation} If there is some $m>0$ such that $u^2\geq m^2>0$ on $M$ and $u$ is of polynomial growth then $u$ is constant.
\end{theorem}
It is also worth to mention that if the Ricci curvature is bounded from below, saying $Ric_M\geq-A$ for some $A\geq0$ and $3u^2-1-A\geq0$, Ratto and Rigoli proved in \cite{RR95} Theorem 2 that $u$ has to be constant if $u$ is of sublinear growth. In fact, using the same argument as in the proof of Theorem \ref{ratto}, we can prove that $u$ must be constant if $u^2\geq m>\frac{A}{2}$ for some $m>0$ and $u$ is of polynomial growth. Therefore, our gradient estimate is a significant improvement of those in \cite{RR95}. It is also nice to inform that the method used in this paper works well for $F=-u^p+u, p>1$ (see section 3 for further details). In this case, we still obtain the same Liouville type theorem as in Theorem \ref{ratto}. On the other hand, we also want to note that Bailesteanu also gave a gradient estimate for Allen-Cahn equation. However, his estimate can not implies Liouville property. Therefore, in some sense, our estimate is better than those in \cite{Bai17}. 

Finally, we introduce another Liouville type theorems for Fisher application as an application of Theorem \ref{thmGE-II}. It is also worth to notice that gradient estimates for Fisher equation on Riemannian manifold was given recently in \cite{CLPW17}. However, we can not use their gradient estimate to show Liouville property.
\begin{theorem}\label{Fisher}
Let $M$ be a Riemannian manifold with $Ric_M\geq0$. Suppose that $u$ is a solution to 
\begin{equation}\label{allen1}
\Delta u-cu(u-1)=0, 
\end{equation}
where $c$ is a positive constant. If there exists $m>0$ such that $u\geq m>0$ on $M$ and $u$ is of polynomial growth then $u$ is constant.
\end{theorem}
This paper is organized as follows. In section 2, we give a proof of the main Theorem \ref{thmGE-II}. The Liouville properties are proven in section 3 as applications of gradient estimates in Theorem \ref{thmGE-II}.
\section{Gradient estimates for the general nonlinear parabolic equation}
To begin with, let us start by giving an important computational lemma.
\begin{lemma}\label{lemm}
Let $ A, B$ be real numbers. Suppose that $u$ is a  solution to the equation 
$$u_t=\Delta_f u+F(u)$$
on $Q_{R, T}:=B(x_0, R)\times[t_0-T, t_0]$, where $x_0\in M$ is a fixed point, $R>0$, and $t_0\in\mathbb{R}$. 
For any $0<\epsilon<1$ fixed, let $w=|\nabla u^{\varepsilon}|^2$, then the following estimate

\begin{equation}
\Delta_fw-w_t\geq -2\max\{K(n-1)+H, 0\}w+2\dfrac{\varepsilon-1}{\varepsilon}\dfrac{1}{u^\varepsilon}\langle\nabla w,\nabla u^\varepsilon\rangle+2\dfrac{1-\varepsilon}{\varepsilon u^{2\varepsilon}}w^2\label{dung1}
\end{equation}
holds true on $Q_{R, T}$ where 
$$
H=(\varepsilon-1)F/u+F'.
$$
\end{lemma}
\begin{proof}Since, $w=|\nabla u^\varepsilon|^2=\varepsilon^2u^{2\varepsilon-2}|\nabla u|^2$, a simple computation shows that
$$\begin{aligned}
 \Delta_fu^\varepsilon
 &=\dfrac{\varepsilon -1}{\varepsilon}\dfrac{w}{u^\varepsilon}+\varepsilon u^{\varepsilon-1}\Delta_f u=\dfrac{\varepsilon -1}{\varepsilon}\dfrac{w}{u^\varepsilon}+\varepsilon u^{\varepsilon-1}(u_t-F).
 \end{aligned}$$

By the Bochner--Weitzenb\"ock formula, we have
\begin{align*}
\Delta_f w\geq&-2K(n-1)w+2\langle\nabla\Delta_f u^\varepsilon,\nabla u^\varepsilon\rangle\\
=&-2K(n-1)w+2\left\langle \nabla\left (\dfrac{\varepsilon -1}{\varepsilon}\dfrac{w}{u^\varepsilon}+\varepsilon u^{\varepsilon-1}(u_t-F)\right ),\nabla u^\varepsilon\right \rangle\\
=&-2K(n-1)w+2\dfrac{\varepsilon-1}{\varepsilon}\dfrac{1}{u^\varepsilon}\langle\nabla w,\nabla u^\varepsilon\rangle+2\dfrac{1-\varepsilon}{\varepsilon u^{2\varepsilon}}w^2+w_t-2\left ((\varepsilon-1)\dfrac{F}{u}+F'\right )w\\
\geq & -2\max\{K(n-1)+H,0\}w+2\dfrac{\varepsilon-1}{\varepsilon}\dfrac{1}{u^\varepsilon}\langle\nabla w,\nabla u^\varepsilon\rangle+2\dfrac{1-\varepsilon}{\varepsilon u^{2\varepsilon}}w^2+w_t
\end{align*}
where $H=(\varepsilon-1)F/u+F'$. The proof is complete.
\end{proof}
\begin{remark}\label{rem}
We have two important remarks.
\begin{enumerate}
\item When $F=-cu^2+cu$, by \eqref{e1}, it is easy to see that 
\begin{align}
H&=(\varepsilon-1)F/u+F'\notag\\
&=c(\varepsilon-1)(-u+1)-2cu+c=c[-u(\varepsilon+1)+\varepsilon].\label{e1}
\end{align}
Hence, $H\leq0$ if $ -u(\varepsilon+1)+\varepsilon\leq 0$. This condition can be rewritten as
$$u\geq \frac{\varepsilon}{\varepsilon+1}.$$
Therefore, if $u\geq m>0$ then we can choose $\epsilon>0$ such that 
$$ m\geq \frac{\varepsilon}{\varepsilon+1}.$$
Thus as its consequence, we have $H\leq0$.
\item Similarly, if $F=-u^3+u$ then 
\begin{align}
H&=(\varepsilon-1)F/u+F'\notag\\
&=(\varepsilon-1)(-u^2+1)-3u^2+1=-(\varepsilon+2)u^2+\varepsilon;\label{e2}
\end{align}
It turns out that $H\leq0$ if $-(\varepsilon+2)u^2+\varepsilon\leq 0$. This condition can be read as
$$u^2\geq \frac{\varepsilon}{\varepsilon+2}. $$
Hence, if $u\geq m>0$, we can choose $\varepsilon$ satisfying 
$$m^2\geq \frac{\varepsilon}{\varepsilon+2}$$
and $H\leq0$.
\end{enumerate}
\end{remark}
\begin{lemma}[see \cite{SZ06, Wu15}]
\label{cutoff}
Fix $t_0\in\mathbb{R}$ and $T>0$. For any give $\tau\in (t_0-T,t_0]$, there exists a smooth function $\bar\psi:[0,+\infty)\times[t_0-T,t_0]\to\mathbb{R}$ satisfying following properties
\begin{enumerate}
\item [(i)] $0\leqslant \bar{\psi}(r,t)\leqslant 1$ 
in $[0,R]\times[t_0-T,t_0]$ and $\bar\psi$ is supported in a subset of $[0,R]\times[t_0-T,t_0]$;
\item [(ii)] $\bar{\psi}(r,t)=1$ and $\bar{\psi}_r(r,t)=0$
in $[0,R/2]\times[\tau,t_0]$ and $[0,R/2]\times[t_0-T,t_0]$, respectively;
\item [(iii)] $|\bar{\psi}_t|\leqslant C (\tau-t_0+T)^{-1}\bar\psi^{1/2} $
in $[0,+\infty)\times[t_0-T,t_0]$ for some $C>0$ and $\bar{\psi}(r,t_0-T)=0$ for all $r\in [0,+\infty)$;
\item [(iv)] $- C_\epsilon\bar{\psi}^{\epsilon} /R \leqslant \bar{\psi}_r\leqslant 0$ and $|\bar{\psi}_{rr}|\leqslant C_\epsilon\bar{\psi}^\epsilon /R^2$
in $[0,+\infty)\times[t_0-T,t_0]$ for every $\epsilon \in (0,1]$ with some constant $C_\epsilon$ depending on $\epsilon$.
\end{enumerate}
\end{lemma}
Now, we give a proof of gradient estimate for the general equation \eqref{general}.
\begin{proof}[Proof of Theorem \ref{thmGE-II}] This proof of this theorem is standard. We will follow the arguments in \cite{DKN18, LY86, SZ06, Wu15, Wuna}. First of all, with each fixed time $\tau\in (t_0-T,t_0]$, we choose a cut-off function $\bar\psi(r,t)$ as in Lemma \ref{cutoff}. To prove the theorem, we will show that the inequality \eqref{theorem1} holds true at every point $(x,\tau)$ in $Q_{R/2,T}$. Let us define the function $\psi:M\times[t_0-T,t_0]\to \mathbb{R}$ given by
$$
\psi (x,t) =\bar\psi(d(x,x_0),t)
$$
where $x_0\in M$ is a fixed point given in the statement of the theorem. Assume that $\psi w$ obtains its maximum value in $\{(x, t)\in M\times[t_0-T, \tau]: d(x, x_0)\leqslant R\}$ at $(x_1,t_1)$.
We may assume that $(\psi w)(x_1, t_1)>0$; otherwise, it follows from $(\psi w)(x_1, t_1)\leqslant 0$ that $(\psi w)(x, \tau)\leqslant 0$ for all $x\in M$ such that $d(x, x_0)\leqslant R$. However, by the definition of $\psi$, we have $\psi(x, \tau)\equiv 1$ for all $x\in M$ satisfying $d(x, x_0)\leqslant R/2$. This implies that $w(x,\tau)\leqslant 0$ when $d(x, x_0)\leqslant R/2$. Since $\tau$ is arbitrary, we conclude that \eqref{theorem1} holds on $Q_{R/2, T}$. Moreover, due to the standard argument of Calabi \cite{Calabi}, we may also assume that $(\psi w)$ is smooth at $(x_1, t_1)$. 

Since $(x_1, t_1)$ is the maximum point of $\psi w$, we infer that at $(x_1,t_1)$, the following facts hold true: $\nabla(\psi w)=0$, $\Delta_f(\psi w)\leqslant 0$, and $(\psi w)_t\geqslant 0$. Hence, still being at $(x_1, t_1)$, we conclude
\[
\begin{split}
0 \geq \Delta_f(\psi w)-(\psi w)_t=\psi(\Delta_fw-w_t)+w(\Delta_f\psi-\psi_t)+2\langle\nabla w,\nabla\psi\rangle.
\end{split}
\]
This inequality combining with \eqref{dung1} implies 
\[
\begin{split}
0\geqslant &-2[(n-1)K+H](\psi w)+2\frac{1-\varepsilon}{\varepsilon u^{2\varepsilon}}\psi w^2\\
&+2\frac{1-\varepsilon}{\varepsilon u^\varepsilon}\left\langle \nabla\psi,\nabla u^\varepsilon\right\rangle w +w\Delta_f\psi-w\psi_t-2\frac{|\nabla \psi|^2}{\psi}w
\end{split}
\]
at $(x_1,t_1)$. In other words, we have just proved that
\begin{equation}\label{24}
\begin{split}
\frac{2(1-\varepsilon)}{\varepsilon u^{2\varepsilon}}\psi w^2\leqslant & 2[(n-1)K+H]\psi w-2\frac{1-\varepsilon}{\varepsilon u^{\varepsilon}}\left\langle\nabla \psi,\nabla u^\varepsilon\right\rangle w \\
&-w\Delta_f\psi+w\psi_t+2\frac{|\nabla\psi|^2}{\psi}w
\end{split}
\end{equation}
at $(x_1,t_1)$. Let $P=\dfrac{1-\varepsilon}{\varepsilon u^{2\varepsilon}}$, then $\dfrac{1}{P}\leq\dfrac{\varepsilon M^{2\varepsilon}}{1-\varepsilon}$ on $Q_{R, T}$. We have two possible cases.

\medskip\noindent\textbf{Case 1}. If $x_1\in B(x_0,R/2)$, then for each fixed $\tau \in (t_0-T, t_0]$, there holds $\psi (\cdot, \tau) \equiv 1$ everywhere on the spacelike in $B(x_0,R/2)$ by the definition of $\psi$. This and \eqref{24} yield
$$
2P\psi w^2\leqslant 2[(n-1)K+H]\psi w+w\psi_t
$$ 
at $(x_1,t_1)$. For arbitrary $(x,\tau) \in B(x_0, R/2) \times (t_0-T, t_0]$, we observe that
\begin{align*} 
w(x, \tau) &=\psi^{1/2}w(x, \tau)\leqslant \frac{\varepsilon M^{2\varepsilon}}{1-\varepsilon}(P\psi^{1/2}w)(x_1,t_1) \\
&\leqslant \frac{\varepsilon M^{2\varepsilon}}{1-\varepsilon}\left\{[(n-1)K+H]\psi^{1/2}\big|_{(x_1,t_1)}+\frac{\psi_t}{2\psi^{1/2}}\Big|_{(x_1,t_1)}\right\}\\ 
&\leqslant \frac{\varepsilon M^{2\varepsilon}}{1-\varepsilon}\left\{[(n-1)K+H]+ c (\tau-t_0 +T)^{-1}\right\},
\end{align*}
where we used Lemma \ref{cutoff} in the first equality and the last inequality. Since $\tau$ can be arbitrarily chosen, we complete the proof of \eqref{theorem1} in this case.

\medskip\noindent\textbf{Case 2}. Suppose that $x_1 \notin B(x_0,R/2)$ where $R\geqslant 2$. From now on, we use $c$ to denote a constant depending only on $n$ whose value may change from line to line. Since $\Ric_f\geqslant -(n-1)K$ and $r(x_1,x_0)\geqslant 1$ in $B(x_0,R)$, we can apply the $f$-Laplacian comparison theorem in \cite{Bri} to obtain
\begin{equation}\label{Bri}
\Delta_fr(x_1)\leqslant \alpha +(n-1)K(R-1), 
\end{equation}
where $\alpha:=\max_{x\in B(x_0,1)}\Delta_fr(x)$. This $f$-Laplacian comparison theorem combining with Lemma \ref{cutoff} implies
\begin{equation}\label{25}
\begin{split}
-w\Delta_f\psi=&-w\big[\psi_r\Delta_fr+\psi_{rr}|\nabla r|^2\big] \\
=&w(-\psi_r)\Delta_fr-w\psi_{rr} \\
\leqslant &-w\psi_r\big(\alpha+(n-1)K(R-1)\big)-w\psi_{rr} \\
\leqslant &w\psi^{1/2}\frac{|\psi_{rr}|}{\psi^{1/2}}+|\alpha|\psi^{1/2}w\frac{|\psi_r|}{\psi^{1/2}}+\frac{(n-1)K(R-1)|\psi_r|}{\psi^{1/2}}\psi^{1/2}w \\
\leqslant &\frac{P}{8}\psi w^2+\frac{c}{P}\Big[\Big(\frac{|\psi_{rr}|}{\psi^{1/2}}\Big)^2+\Big(\frac{|\alpha||\psi_rw|}{\psi^{1/2}}\Big)^2+\Big(\frac{K(R-1)|\psi_r|}{\psi^{1/2}}\Big)^2\Big] \\
\leqslant &\frac{P}{8}\psi w^2+\frac{c}{P}\Big(\frac{1}{R^4}+\frac{|\alpha|^2}{R^2}+K^2\Big).
\end{split}
\end{equation}
On the other hand, by the Young inequality, we obtain
\begin{equation}\label{26}
\begin{split}
-2\frac{1-\varepsilon}{\varepsilon u^\varepsilon}\left\langle\nabla\psi,\nabla u^\varepsilon\right\rangle w 
\leqslant  &2\big[P\psi w^2\big]^{3/4}\frac{(1-\varepsilon)|\nabla\psi|}{\varepsilon u^\varepsilon(P\psi)^{3/4}} \\
\leqslant &P\psi w^2+\frac{c}{P}\dfrac{(1-\varepsilon)^2}{\varepsilon^2}\dfrac{|\nabla\psi|^4}{\psi^3}\\
\leqslant &P\psi w^2+\frac{c}{P}\frac{(1-\varepsilon)^2}{\varepsilon^2}\dfrac{1}{R^4}.
\end{split}
\end{equation}
By using the Cauchy--Schwarz inequality several times, it is not hard for us to see that the following estimates hold true: first with note that $0\leq\psi\leq1$, we have
\[
2[(n-1)K+H]\psi w\leqslant \frac{P}{8}\psi w^2+ \frac{c}{P}[(n-1)K+H]^2
\]
then 
\[\begin{split}
w\psi_t=\psi^{1/2}w\frac{\psi_t}{\psi^{1/2}}&\leqslant \frac{P}{8}\psi w^2+\frac{c}{P}\Big(\frac{\psi_t}{\psi^{1/2}}\Big)^2 \leqslant \frac{P}{8}\psi w^2+\frac{c}{P}\frac{1}{(\tau -t_0+T)^2}
\end{split}\]
and finally for $|\nabla\psi|^2 w/\psi$, we obtain
\[\begin{split}
2\frac{|\nabla\psi|^2}{\psi}w
=2\psi^{1/2}w\frac{|\nabla\psi|^2}{\psi^{3/2}}&\leqslant \frac{P}{8}\psi w^2+\frac{c}{P}\Big(\frac{|\nabla\psi|^2}{\psi^{3/2}}\Big)^2 \leqslant \frac{P}{8}\psi w^2+\frac{c}{P}\frac{1}{R^4}. 
\end{split}\]
Now, we combine \eqref{24}--\eqref{26} and all above three estimates to get
\[\begin{split}
2 P\psi w^2\leqslant & P\psi w^2+\frac{c}{P} \left\{
\begin{split}
&\frac{1}{R^4}+\frac{(1-\varepsilon)^2}{\varepsilon^2}\frac{1}{R^4}+\frac{\alpha^2}{R^2} \\
&+[(n-1)K+H]^2+\frac{1}{(\tau -t_0+T)^2}
\end{split}
\right\} +\frac{P}{2}\psi w^2.
\end{split}\]
This inequality implies 
\begin{align*}
\psi w^2\leqslant \frac{c}{P^2}\left(\frac{1}{R^4}+\frac{(1-\varepsilon)^2}{\varepsilon^2}\frac{1}{R^4}+[(n-1)K+H]^2+\frac{\alpha^2}{R^2}+\frac{1}{(\tau -t_0+T)^2}\right).
\end{align*} 
The finally, since $\psi(\cdot,\tau) \equiv 1$ in $B(x_0,R/2)$ , we conclude that
$$\begin{aligned}
w^2(x, \tau)&\leqslant \psi w^2(x_1,t_1)\\
&\leqslant \frac{c\varepsilon^2M^{4\varepsilon}}{(1-\varepsilon)^2}\left(\frac{1}{R^4}+\frac{(1-\varepsilon)^2}{\varepsilon^2}\frac{1}{R^4}+\frac{\alpha^2}{R^2}+\frac{1}{(\tau -t_0+T)^2}+[(n-1)K+H]^2\right)
\end{aligned}$$
for all $x\in B(x_0, R/2)$. Since $\tau$ is arbitrary, this also completes the proof of Theorem \eqref{theorem23} in this case.
\end{proof}
Now, we show an application to study Liouville properties of heat solutions on smooth metric measure space $(M, g, e^{-f}dv)$. For some $a\in\mathbb{R}$, let $F=au\log u$ and consider the following heat equation
\begin{equation}\label{grs}
u_t=\Delta_fu+au\log u.
\end{equation}
\begin{proof}[Proof of Corollary \ref{corol1}]
For any $0<\varepsilon<1$, since $F=au\log u$, we can compute 
$$\begin{aligned}
H&=(\varepsilon-1)F/u+F'
=a(\varepsilon-1)\log u+a(\log u+1)\\
&=a(1+\varepsilon\log u).
\end{aligned}$$
\begin{enumerate}
\item If $a>0$ and $0<u\leq c<e^{-1}$, we can choose $0<\varepsilon<1$ such that $H\leq0$. We fix such $\varepsilon$. By Theorem \ref{thmGE-II}, we have that
\begin{equation}\label{w1}
\varepsilon u^{\varepsilon-1}|\nabla u|\leq c(n)M^\varepsilon\sqrt{\frac{\varepsilon}{1-\varepsilon}}\left (\frac{1}{R}+\sqrt{\frac{1-\varepsilon}{\varepsilon}}\frac{1}{R}+\frac{\sqrt{\alpha}}{\sqrt{R}}+\frac{1}{\sqrt{t-t_0+T}}\right)
\end{equation}
for all $(x, t)\in Q_{R, T}$. Let $R, T\to\infty$, we conclude that $|\nabla u|\equiv 0$. Therefore $u_t=au\log u$, this implies $u=exp(de^{at})$. It is easy to see that $\lim\limits_{t\to-\infty}u=1$. This is a contradiction since $u\leq c<e^{-1}<1$.
\item If $a<0$ and $0<c\leq u(x, t)\leq D<1$, we also can choose $0<\varepsilon<1$ such that $H\leq0$. Using the same argument as in the proof of part (1), we have that $|\nabla u|\equiv0$. This implies a contradiction since $u\leq D<1$. Now assume that $u(x, t)=o(r^{N_1}(x)+t^{N_2})$ for some $N_1, N_2\in \mathbb{N}$. In this cases, we can choose $\varepsilon>0$ small enough such that $H\leq0, N_1\varepsilon, N_2\varepsilon<1/2$. Let $M:=\sup\limits_{Q_{R,T}}u$ Since
$$\lim\limits_{R\to\infty}\frac{M^\varepsilon}{\sqrt[]{R}}=\lim\limits_{R\to\infty}\frac{o(R^{\varepsilon N_1})}{\sqrt[]{R}}=0=\lim\limits_{T\to\infty}\frac{M^\varepsilon}{\sqrt[]{T}}=\lim\limits_{T\to\infty}\frac{o(T^{\varepsilon N_2})}{\sqrt[]{T}}$$
Letting $R, T\to\infty$ in \eqref{w1}, we conclude that $|\nabla u|\equiv0$. As in the proof of the part (1), we infer $u=exp(de^{at})$. If $d=0$, we have $u\equiv 1$. If $d<0$, let $t\to-\infty$, we have $u\to 0$, this is impossible since $u\geq c>0$. If $d>0$, we see that $u$ is of exponent growth in $t$, this is also a contradiction.
\item If $a=0$ the obviously $H=0$, we argue as in the proof of part (2) to conclude that $|\nabla u|\equiv0$. Hence $\Delta_fu=0$, thus $u_t=0$. This implies $u$ is constant. 
\end{enumerate}
The proof is complete.
\end{proof}
It is worth to note that Corollary \ref{corol1} is a signficant improvement of a Wu's main result (see \cite{Wuna} Theorem 1.3). In fact, in \cite{Wuna}, the author required that the bounds of $u$ is $e^{-2}$ or $u$ is of sublinear growth in distance and of square root growth in time. However, when $a=0$, it is proven in \cite{DKN18, Wu15} that we can raise the growth of $u$ to exponent rate. The below example shows that we can not require $u$ is of exponent growth. Hence, the results given in \cite{DKN18, Wu15} are optimal.
\begin{example}Let $(\mathbb{R}^n, g)$ be a Riemannian manifold with standard metric $g$ and $f=ax_1$. We choose $u=e^{ax_1}$. It is easy to see that $u$ is not of polynomial growth, positive and satisfies $u_t=\Delta_fu$.
\end{example}
\section{Liouville type theorems for the Allen - Cahn and Fisher equations}
Let us consider the Allen - Cahn equation
\begin{equation}\label{allen}
u_t=\Delta_f u-u^3+u, 
\end{equation}
Note that $u$ is a standing solution if and only if $u_t=0$. 
\begin{theorem}\label{main1}Let $(M, g, e^{-f}dv)$ is a smooth metric measure space with $Ric_f\geq 0$. Assume that $u$ is a bounded solution of the following equation
$$u_t=\Delta_fu-u^3+u.$$
If $u$ is a standing solution and one of the following conditions holds true
\begin{enumerate}
\item $u^2\geq m^2>0$ on $M$, for some $m>0$, 
\item $u>0$, and $\inf\limits_{Q_{R,T}}u=o\left(\dfrac{1}{R^{\theta_2}+t^{\theta_1}}\right)$, for some $0<\theta_1, \theta_2<1/2$.
\end{enumerate}
then $u$ is constant.
\end{theorem}
\begin{proof}
First, we assume that $u^2\geq m^2>0$, for some $m>0$. By the boundedness of $u$, we can suppose that $|u|\leq C$ for some $C>0$. Since $u^2\geq m^2>0$ and $M$ is connected, we infer either $u\geq m$; or $u\leq-m$. Observe that if $u$ is a solution to \eqref{allen} then $-u$ also satisfies \eqref{allen}. Hence, we may assume that $u\geq m>0$. As in Remark \ref{rem}, we may choose $\varepsilon>0$ small enough such that 
$$u^2\geq {m}^2\geq\frac{\varepsilon}{\varepsilon+2}.$$
It turns out that $-(2+\varepsilon)u^2+\varepsilon \leq0$. Consequently, 
$$\max\limits_{Q_{R,T}}\left\{-(\varepsilon+2)u^2+\varepsilon, 0\right\}=0$$

Using Theorem \ref{thmGE-II} with $K=0$, we have 
\[
\varepsilon u^{\varepsilon-1}|\nabla u|\leq c(n)C^\varepsilon\sqrt{\frac{\varepsilon}{1-\varepsilon}}\left (\frac{1}{R}+\sqrt{\frac{1-\varepsilon}{\varepsilon}}\frac{1}{R}+\frac{\sqrt{\alpha}}{\sqrt{R}}+\frac{1}{\sqrt{t-t_0+T}}\right ).
\]
This implies
\begin{equation}\label{cond2}
|\nabla u|\leq c(n)C\sqrt{\frac{1}{\varepsilon(1-\varepsilon)}}\left (\frac{1}{R}+\sqrt{\frac{1-\varepsilon}{\varepsilon}}\frac{1}{ R}+\frac{\sqrt{\alpha}}{\sqrt{R}}+\frac{1}{\sqrt{t-t_0+T}}\right ).
\end{equation}
Letting $R\to \infty$, then letting $t\to\infty$, it turns out that
$|\nabla u|=0$. Therefore $u$ is constant since $u_t=0$.

Finally, assume that $u>0$ and $\lim\limits_{R\to\infty}\inf\limits_{B(x_0, R)}u=0$. Since $u>0$, we have $m=\inf\limits_{Q_{R, T}}u\not=0$ and $\lim\limits_{R\to\infty}m=0$. Note that on $Q_{R,T}$, we have 
$$m^2>\frac{\varepsilon}{2+\varepsilon}$$
for some $\varepsilon>0$ small enough, saying $\varepsilon:=m^2$. Moreover, it is easy to see that 
$$\lim\limits_{R\to\infty}\frac{1}{Rm^2}=0=\lim\limits_{t\to\infty}\frac{1}{tm^2}.$$
Here we used $\lim\limits_{t\to\infty}t^{\theta_1} m=\lim\limits_{R\to\infty}R^{\theta_2} m=1$. Combining this observation and \eqref{cond2}, we implies that $|\nabla u|=0$. Hence $u$ is constant since $u_t=0$.
\end{proof}
\begin{proof}[Proof of Theorem \ref{ratto}]By the same argument as in the proof of Theorem \ref{main1}, we may assume that $u\geq m>0$. Consequently, 
$$\max\limits_{Q_{R, T}}\left\{-(\varepsilon+2)u^2+\varepsilon, 0\right\}=0$$ 
Now, by assumption on the growth of $u$, we have $M_R:=\sup\limits_{B(x_0, R)}u=o(R^N)$, for some $N>0$. We choose $\varepsilon>0$ small enough such that $N\varepsilon<1$ then we fix $\varepsilon$. Since $Ric_M\geq0$, using Laplacian comparison theorem as in \cite{LY86} instead of \eqref{Bri}, we may assume that $\alpha=1/R$. Using the inequality \eqref{theorem1} in Theorem \ref{thmGE-II}, we obtain
$$
\varepsilon u^{\varepsilon-1}|\nabla u|\leq c(n, \varepsilon)\frac{M_R^\varepsilon}{R}\left (1+\sqrt{\frac{1-\varepsilon}{\varepsilon}}\right)\leq c(n, \varepsilon)\frac{o(R^{N\varepsilon})}{R}\left (1+\sqrt{\frac{1-\varepsilon}{\varepsilon}}\right)
$$
on $B(x_0, R)$. Fix $y\in B(x_{0}, R)$, then letting $R\to\infty$, we obtain $|\nabla u|(y)=0$. Thus $|\nabla u|\equiv0$ since $y$ is arbitrary. Hence, we conclude that $u$ is constant. The proof is complete.
\end{proof}
\begin{proof}[Proof of Theorem \ref{Fisher}]
Since $u\geq m>0$, as in Remark \ref{rem} we may choose $\varepsilon>0$ such that 
$$\max\limits_{Q_{R,T}}\left\{c[-u(\varepsilon+1)+\varepsilon], 0\right\}=0.$$
Using the same argument as in the proof of Theorem \ref{ratto}, we are done.
\end{proof}
Finally, if we let $F=u^{q}-u^p$, for $p>q\geq1$ then we obtain the following Liouville property.
\begin{corollary}Let $M$ be a Riemannian manifold with $Ric_M\geq0$ and $u$ be a positive solution to the equation
$$\Delta u+u^q-u^p=0.$$
Assume that $u^{p-q}\geq c>\frac{q-1}{p-1}$ and $u$ is of polynomial growth then $u$ is consant.
\end{corollary}
\begin{proof}Using the same argument as in the proof of Theorem \ref{ratto}, we only need to show that for any $0<\epsilon<1$ small enough, $H\leq0$. We have
$$\begin{aligned}
H&=(\varepsilon-1)F/u+F'=(\varepsilon-1)(u^{q-1}-u^{p-1})+qu^{q-1}-pu^{p-1}\\
&=u^{q-1}(q+\varepsilon-1)-u^{p-1}(p+\varepsilon-1).
\end{aligned}$$
Since $u^{p-q}\geq c>\frac{q-1}{p-1}$, and using the fact that 
$$\lim\limits_{\varepsilon\to0}\frac{q+\varepsilon-1}{p+\varepsilon-1}=\frac{q-1}{p-1},$$
we conclude that $H\leq0$ for any $0<\varepsilon$ small enough. The proof is complete.
\end{proof}
\section*{Acknowlegement} 
The first author was partially supported by NAFOSTED under grant number 101.02-2017.313.


\begin{thebibliography}{99}
\bibitem[Bai17]{Bai17} \textsc{M. Bailesteanu}, \emph{A Harnack inequality for the parabolic Allen-Cahn equation}, Ann. Glob. Anal. Geom., \textbf{51} (2017) 367 - 378.

\bibitem[Bri13]{Bri} 
\textsc{K. Brighton}, 
\emph{A Liouville-type theorem for smooth metric measure spaces}, Jour. Geom. Anal., \textbf{23} (2013), 562--570.

\bibitem[CCK15]{CCK15}, 
\textsc{X. D. Cao, M. Cerenzia, and D. Kazaras}, \emph{Harnack estimate for the endangered species equation}, Proc. Amer. Math. Soc, \textbf{143} (2015), 4537 - 4545.

\bibitem[CLPW17]{CLPW17} 
\textsc{X. D. Cao, B. W. Liu, I. Pandleton, and A. Ward}, \emph{Differential Harnack estimates for Fisher's equation}, Pacific Jour. Math., \textbf{290} (2017) 273 - 300.

\bibitem[Cal57]{Calabi} 
\textsc{E. Calabi}, 
\emph{An extension of E. Hopf's maximum principle with an application to Riemannian geometry}, Duke Math. Jour., \textbf{25} (1957), 45-56.

\bibitem[dPKW13]{dPKW13} 
\textsc{M. del Pino, M. Kowalczyk, and J. C. Wei}, \emph{Entire solutions of the Allen-Cahn equation and complete embedded minimal surfaces of finite total curvature in $\mathbb{R}^3$},
Jour. Diff. Geom., \textbf{93} (2013) 67 - 131.


\bibitem[DKN18]{DKN18} 
\textsc{N. T. Dung, N. N. Khanh, and N. Q. Anh}, 
\emph{Gradient estimates for some $f$-heat equations driven by Lichnerowicz’s equation on complete smooth metric measure spaces}, 
Manuscripta Math., \textbf{155} (2018), 471 - 501. 

\bibitem[Fis37]{fisher} 
\textsc{R. A. Fisher}, \emph{The wave of advance of advantageous genes}, Annals of Eugenics, \textbf{7} (1937), 355 - 369,
 
\bibitem[KPP]{Kol} 
\textsc{A. N. Kolmogorov, I. G. Petrovskii, and N. S. Piskunov}, \emph{Etude de l'\'equation de la diffusion avec croissance de la quantit\'e de mati\'ere et son application \'a un probl\'eme biologique}, Bulletin Universit\'e d’Etat\'a Moscou, 1 - 26, S\'erie internationale, section A 1, 1937

\bibitem[LY86]{LY86}
\textsc{P. Li and S. T. Yau}, \emph{On the parabolic kernel of the Schr\"odinger operator} Acta Math., \textbf{156} (1986) (3-4) 153 - 201,

\bibitem[Pac12]{Pac12} 
\textsc{F. Pacard}, \emph{The role of minimal surfaces in the study of the Allen-Cahn equation}, In Geometric analysis: partial differential equations and surfaces, volume 570 of Contemp. Math (2012), pages 137 - 163. Amer. Math. Soc.,
Providence, RI.

\bibitem[RR95]{RR95} \textsc{A. Ratto and M. Rigoli}, \emph{Gradient bounds and Liouville's type theorems for the Poisson equation on complete Riemannian manifolds}, Tohoku Math. Jour., \textbf{47} (1995) 509 - 519.

\bibitem[SZ06]{SZ06} 
\textsc{P. Souplet, Q.S. Zhang}, 
\emph{Sharp gradient estimate and Yau's Liouville theorem for the heat equation on noncompact manifolds}, Bull. London Math. Soc., \textbf{38} (2006), 1045-1053.

\bibitem[Wu15]{Wu15} 
\textsc{J. Y. Wu}, 
\emph{Elliptic gradient estimates for a weighted heat equation and applications}, Math. Zeits., \textbf{280} (2015), 451-468. 

\bibitem[Wu17]{Wuna}
\textsc{J. Y. Wu}, \emph{Elliptic gradient estimates for a nonlinear heat equation and applications}, Nonlinear Analysis: TMA, \textbf{151} (2017), 1-17.

\end{thebibliography}
\end{document}